\documentclass{spie}

\usepackage{graphicx}

\title{\Large \bf Quantum Knots and Knotted Zeros}

\author{Louis H. Kauffman\supit{a} and Samuel J. Lomonaco Jr.\supit{b}
\skiplinehalf
\supit{a} Department of Mathematics, Statistics and Computer Science  
(m/c 249), 851 South Morgan Street, University of Illinois at Chicago,
Chicago, Illinois 60607-7045, USA \\and\\Department of Mechanics and Mathematics,Novosibirsk State University,Novosibirsk, Russia\\
\supit{b} Department of Computer Science and Electrical Engineering, University of
Maryland Baltimore County, 1000 Hilltop Circle, Baltimore, MD 21250, USA}
 
\authorinfo{Further author information: L.H.K.  E-mail: kauffman@uic.edu,
S.J.L. Jr.: E-mail: lomonaco@umbc.edu}
  
\begin{document} 

\newcommand{\Across}{\raisebox{-0.25\height}{\includegraphics[width=0.5cm]{A.eps}}}
\newcommand{\Bcross}{\raisebox{-0.25\height}{\includegraphics[width=0.5cm]{B.eps}}}
\newcommand{\Asmooth}{\raisebox{-0.25\height}{\includegraphics[width=0.5cm]{C.eps}}}
\newcommand{\Bsmooth}{\raisebox{-0.25\height}{\includegraphics[width=0.5cm]{D.eps}}}
\newcommand{\Rcurl}{\raisebox{-0.25\height}{\includegraphics[width=0.5cm]{Rcurl.eps}}}
\newcommand{\Lcurl}{\raisebox{-0.25\height}{\includegraphics[width=0.5cm]{Lcurl.eps}}}
\newcommand{\Arc}{\raisebox{-0.25\height}{\includegraphics[width=0.5cm]{Arc.eps}}}

 \maketitle

\begin{abstract}
In 2001, Michael Berry \cite{Berry} published the paper "Knotted Zeros in the Quantum States of Hydrogen" in Foundations of Physics. In this paper we show how to place Berry's discovery in the context of general knot 
theory and in the context of our formulations for quantum knots. Berry gave a time independent wave function for hydrogen, as a map from three space $R^3$ to the complex plane and such that the inverse image of $0$ in the complex plane contains a knotted curve in $R^3.$ We show that for knots in $R^3$ this is a generic situation in that every smooth knot $K$ in $R^3$ has a smooth classifying map $f:R^3 \longrightarrow C$ (the complex plane) such that  $f^{-1}(0) = K.$ This leaves open the question of characterizing just when such $f$ are wave-functions for quantum systems. One can compare this result with the work of Mark Dennis and his collaborators, with the work of Daniel Peralta-Salas and his collaborators, and with the work of Lee Rudolph. Our approach provides great generality to the structure of knotted zeros of a wavefunction and opens up many new avenues for research in the relationships of quantum theory and knot theory. We show how this classifying construction can be related our previous work on two dimensional and three dimensional mosaic and lattice quantum knots.
\end{abstract}


\section{Introduction}
The purpose of this paper is to place our concept of quantum knots in a framework that includes knots that are given a ``classifying map" $f: S^3 \longrightarrow C$ where $C$ is the
complex plane, and the knot $K = f^{-1}(0)$ is the inverse image of the origin in the complex plane. That is, the knot or link is the set of zeroes of the "wavefunction" $f.$ This is the vision of knots and their relation to wavefunctions that is proposed by Michael Berry \cite{Berry}.  We begin in Section 2 by reviewing
 our previous work on quantum knots where we model the topological information in a knot by a state
vector in a Hilbert space that is directly constructed from mosaic diagrams for the knots. In Section 3 we
 give a general definition of quantization of mathematical structures and apply it to the quantization of 
 the set of classical knots (embeddings of a circle into three dimensional space). The group of 
 homeomorphisms of three dimensional space acts on this set of embeddings. The Hilbert space that results from this set of embeddings is very large, but descriptive of the sort of knotting phenomena that may occur in nature such as knotted vortices in super-cooled Helium or knotted gluon fields.
 In Section 4 we define classifying maps for knots as described above, prove that all knots can be described by such mappings and discuss the contexts, topological and physical that are relevant to further work in the direction of this paper. The present paper is meant to be a first step in connecting our formulations of quantum knots with the context of knotted zeroes of quantum wavefunctions. For more about the basic aspects of our quantization procedures the reader is referred to \cite{LomQknots1,LomQknots2,KauffQknots,KauffQknots1}.  

\noindent{\bf Acknowledgement.} Kauffman's work was supported by the Laboratory of Topology and Dynamics, Novosibirsk State University (contract no. 14.Y26.31.0025 with the Ministry of Education and Science of the Russian Federation). Lomonaco's work was supported by NASA Grant Number NNH16ZDA001N-AIST16-0091.\\

\section{Mosaic Quantum Knots}
We begin by explaining the basic idea of mosaic quantum knots as it appears  in 
our papers \cite{KauffQknots,KauffQknots1,LomQKnots,LomQknots1,LomQknots2}. An application
of quantum knots can be found \cite{Shor} in the paper by Farhi, Gosset, Hassidim, Lutominski and 
Shor. There the reader will find proof that quantum knots is a money-making idea.
\bigbreak

View Figures 1, and 2. In the left-most part of Figure 1 we illustrate a mosaic version of a trefoil knot using a $4 \times 4$ space of tiles. In Figure 2 we show the eleven basic 
tiles that can be repeated used in $n \times n$ tile spaces to make diagrams for any classical
knot or link. So far this is a method for depicting knots and links and has no quantum interpretation.
However, as in our previous papers, we use the philosophy that given a well-defined discrete set of 
objects, one can define a vector space with an orhonormal basis that is in one-to-one correspondence
with these objects. Here we let $V$ be the complex vector space with basis in one-to-one 
correspondence with the set of eleven basic tiles shown in Figure 2. An $n \times n$ mosaic as shown in Figure 1 is then regarded as an element in the tensor product of $n^{2}$ copies of $V$.
We order the tensor product by consectively going through the rows of the mosaice from left to right 
and from top to bottom. In this way, knot diagrams represented by $n \times n$ mosaics are
realized as vectors in $H_{n} = V \otimes V \otimes \cdots \otimes V$ where there are 
$n^{2}$ factors in this tensor product.
\bigbreak

Isotopy moves on the mosaic diagrams are encoded by tile replacements that induce unitary transformations on the Hilbert space. We refer to \cite{LomQKnots} for the details. The upshot of this fomulation of 
isotopies of the knots is that the diagrammatic isotopies correspond to unitary transformations of the 
Hilbert space $H_{n}$ when the isotopies are restricted to the $n \times n$ lattice. In this way we obtain for each $n$ a group of isotopies $A_{n}$ that we call the {\em ambient  group}. This has the advantage that it turns a version of the Reidemeister moves on knot and link diagrams into a group
and it provides for a quantum formulation not just for the knot and link diagrams, but also for their 
isotopies. Knots and links are usually regarded as entities of a classical nature. By 
making them into quantum information, we have created a domain of quantum knots and links.
\bigbreak

There are many problems and many avenues available for the exploration of quantum knots and links.
It is not the purpose of this paper to specialize in this topic. We show them in order to emphasize the idea that one can quantize combinatorial categories by formulating appropriate Hilbert spaces for their
objects and morphisms. However, it is worth mentioning that other diagrammatic categories for knots and links can be easily accomodated in the mosaic link framework. View Figure 1 again and examine the middle and right diagrams in the figure. Here we show diagrams containing white and black graphical nodes. These can be interpreted for extensions of knot theory to virtual knots or knotted graph 
embeddings. One extends the vector space for the basic tiles and then also adds moves that are
appropriate for the theory in question. In the case of virtual knot theory and the theory of knotted graphs, there is no problem in making these extensions. We will carry them out in a separate paper. The point of 
this section has been to remind the reader of our previous work, and to point to these avenues along which it can be extended.  \bigbreak

Lets go back to classical knot theory in mosaic form. One of the problems in studying this theory is the 
matter of articulating invariants of knots so that they are quantum observables for the theory.
Many invariants such as the Jones polynomial and even the bracket model for the Jones polynomial seem to be resistant to formulation as quantum 
observables. However, there is a very general result for mosaic quantum knots that is intellectually satisfying that we have proved in our earlier work \cite{LomQKnots}.
\bigbreak

\begin{theorem} Let $| K \rangle$ be a mosaic knot diagram in an $n \times n$ lattice.
Then there is a quantum observable $\chi(K)$ such that $\chi(K) | K' \rangle = | K' \rangle $ if and only if
$ | K' \rangle $ is in the orbit of $ | K \rangle $ under the action of the ambient  group
$A_{n}.$ When $|K' \rangle$ is not in the orbit, then $\chi(K) | K' \rangle = 0.$ In other words $\chi(K)$ is a characteristic function for the knot-type of $K$ in the $n \times n$
mosaic lattice.
\end{theorem}

\begin{proof} Define $\chi(K)$ by the formula
$$\chi(K) = \sum_{|K' \rangle \in Orbit(K)} | K' \rangle \langle K' |$$
where $Orbit(K)$ denotes the orbit of $| K \rangle$ under the action of the ambient 
group $A_{n}.$ Note that $Orbit(K)$ is a finite set. The Theorem follows directly from this 
definition.
\end{proof}

This Theorem does not make invariants of knots that are efficient to calculate, but it is intellectually 
satisfying to know that, in principle, in the $n \times n$ lattice, we can distinguish two diagrams that
are inequivalent by the Reidemeister moves for that lattice size. Furthermore, we can use the 
characteristic observables $\chi(K).$  to make observables for any real valued classical knot invariant.
By a classical knot invariant, we mean a function on standard knot diagrams that is invariant under the
usual graphical Reidemeister moves. Such a function is also tautologically defined on mosaic diagrams
and is invariant under the mosaic moves for any $n \times n$ mosaic lattice. For example, the 
Jones polynomial \cite{JO} $V_{K}(t)$ is a Laurent polynomial valued invariant. 
By taking the variable $t$
 in the Jones polynomial to be a specific real number, we obtain from the Jones polynomial, a real-valued classical knot invariant.
\bigbreak

\begin{theorem}Let $Inv(K)$ denote a real-valued classical invariant of knots and links.
Then there is an observable $O$ on the Hilbert space for any $n \times n$ mosaic lattice such 
that $O|K \rangle = Inv(K)|K \rangle$ for any knot vector $| K \rangle$ in the lattice. In this sense,
any real-valued classical knot invariant corresponds to a quantum observable whose eigenvalues
are the values of this invariant.
\end{theorem}

\begin{proof} Define the observable $O$ by the formula
$$O = \sum_{K} Inv(K) \chi(K)$$ where $K$ runs over one representative for each ambient group  orbit in  the $n \times n$ mosaic lattice. Here $\chi(K)$ is the observable defined in Theorem 1. The 
Theorem then follows directly from this definition. 
\end{proof}

In this sense, the quantum observables for mosaic quantum knots are universal with respect to 
real-valued classical knot invariants. It remains to be seen if there are such observables that have
a better than classical efficiency of calculation.
\bigbreak

Other issues for quantum knots involve considering superpositions of them and properties of these
superpositions. We refer the reader to \cite{KauffQknots,LomQKnots,LomQknots1} for examples along these lines.
\bigbreak

 \begin{figure}[htb]
     \begin{center}
     \begin{tabular}{c}
$ 
\begin{array}{cccc}
\includegraphics[width=1cm]{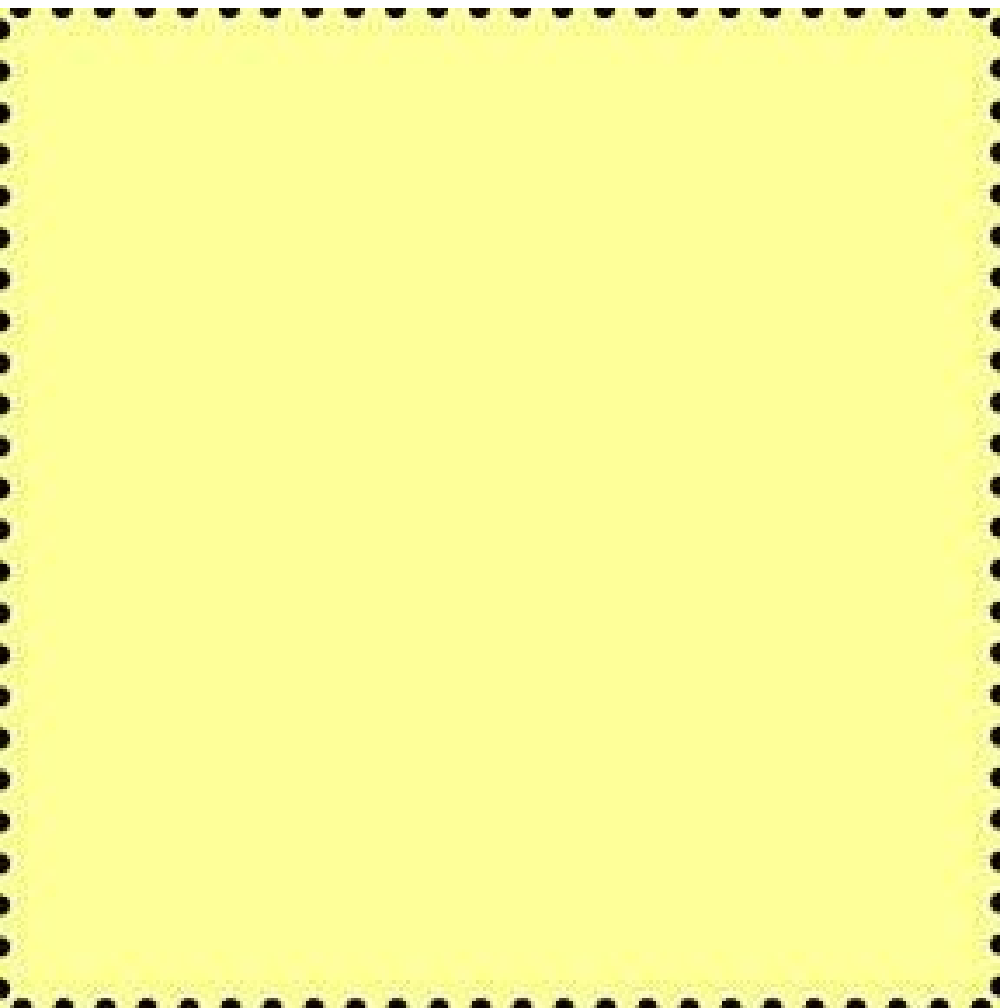} & \includegraphics[width=1cm]{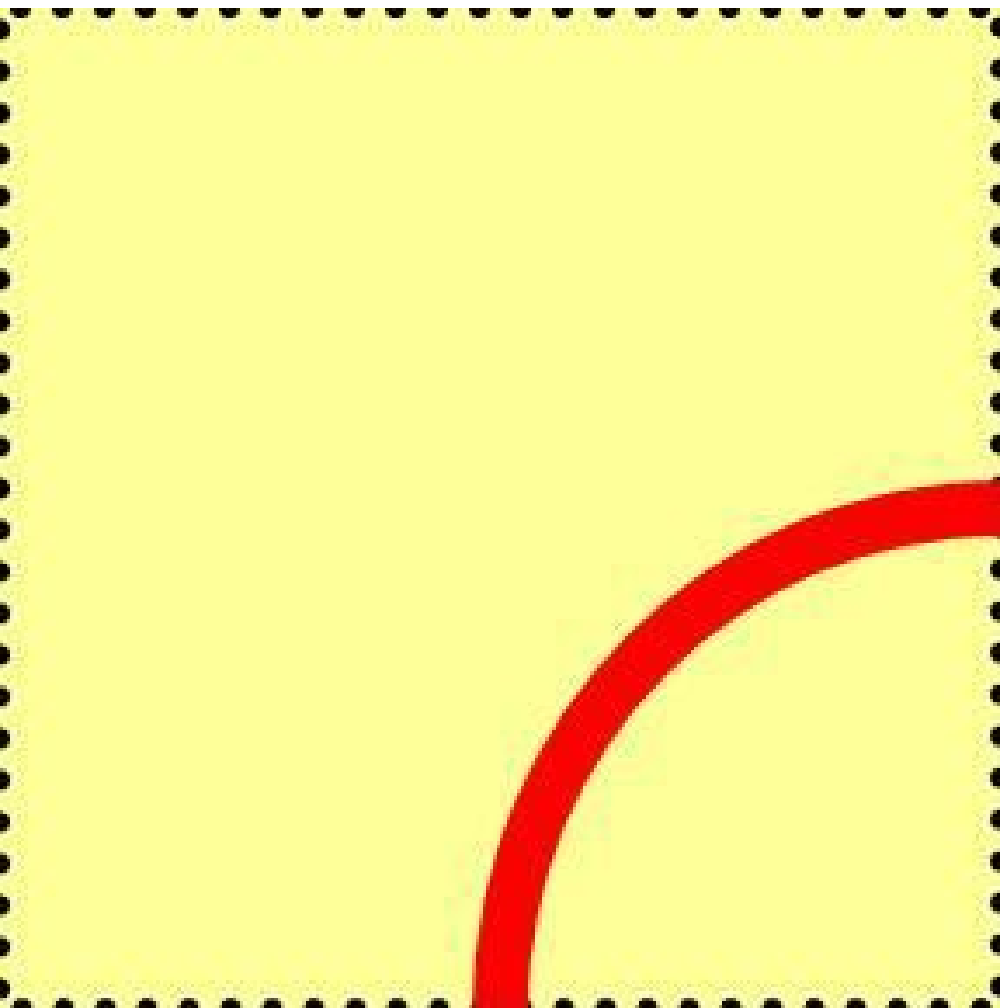} & 
\includegraphics[width=1cm]{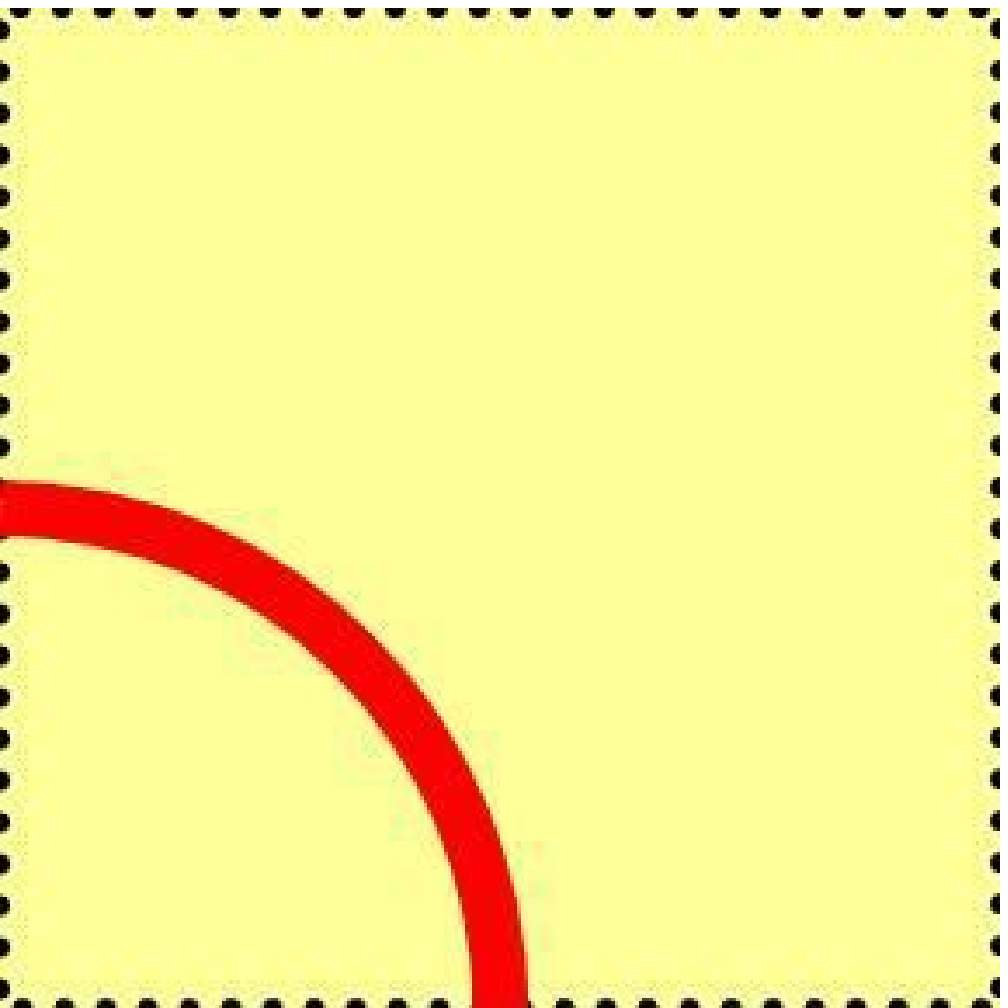} &\includegraphics[width=1cm]{ut00.EPS}  \\ 
\includegraphics[width=1cm]{ut02.EPS} & \includegraphics[width=1cm]{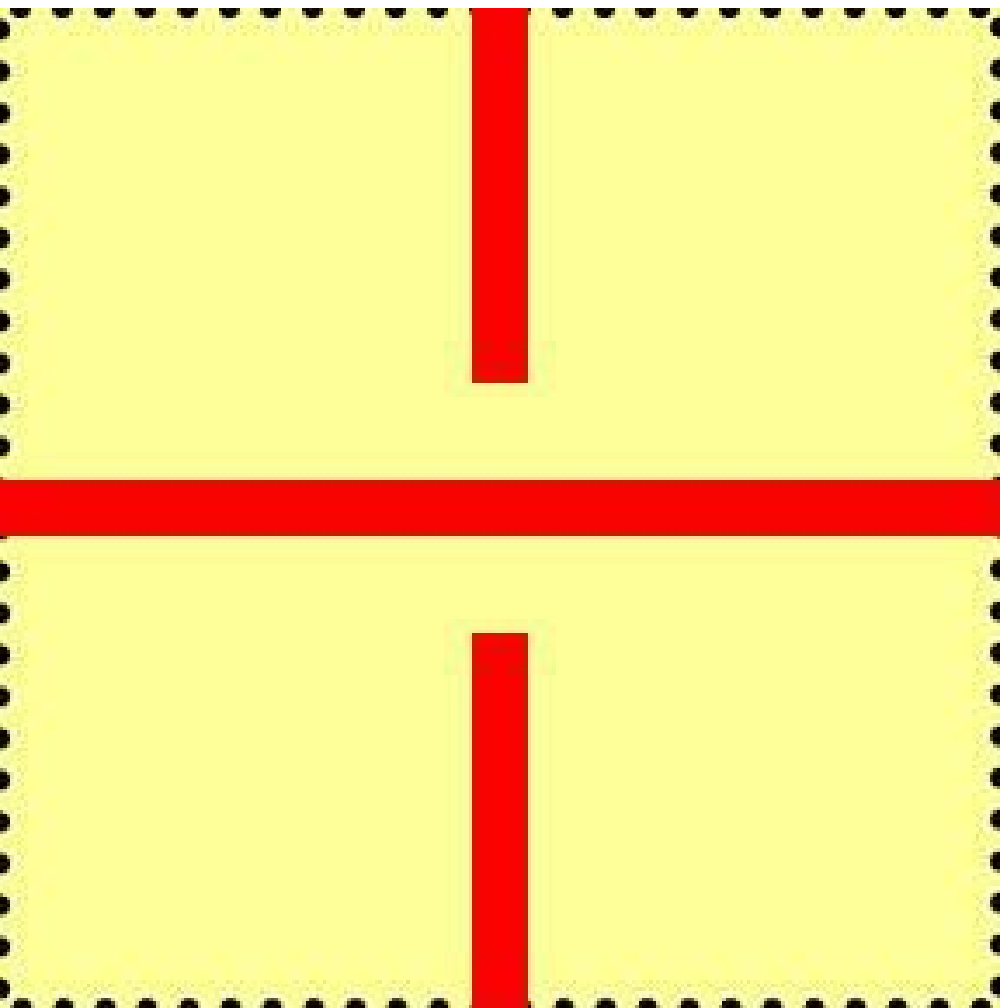} &
\includegraphics[width=1cm]{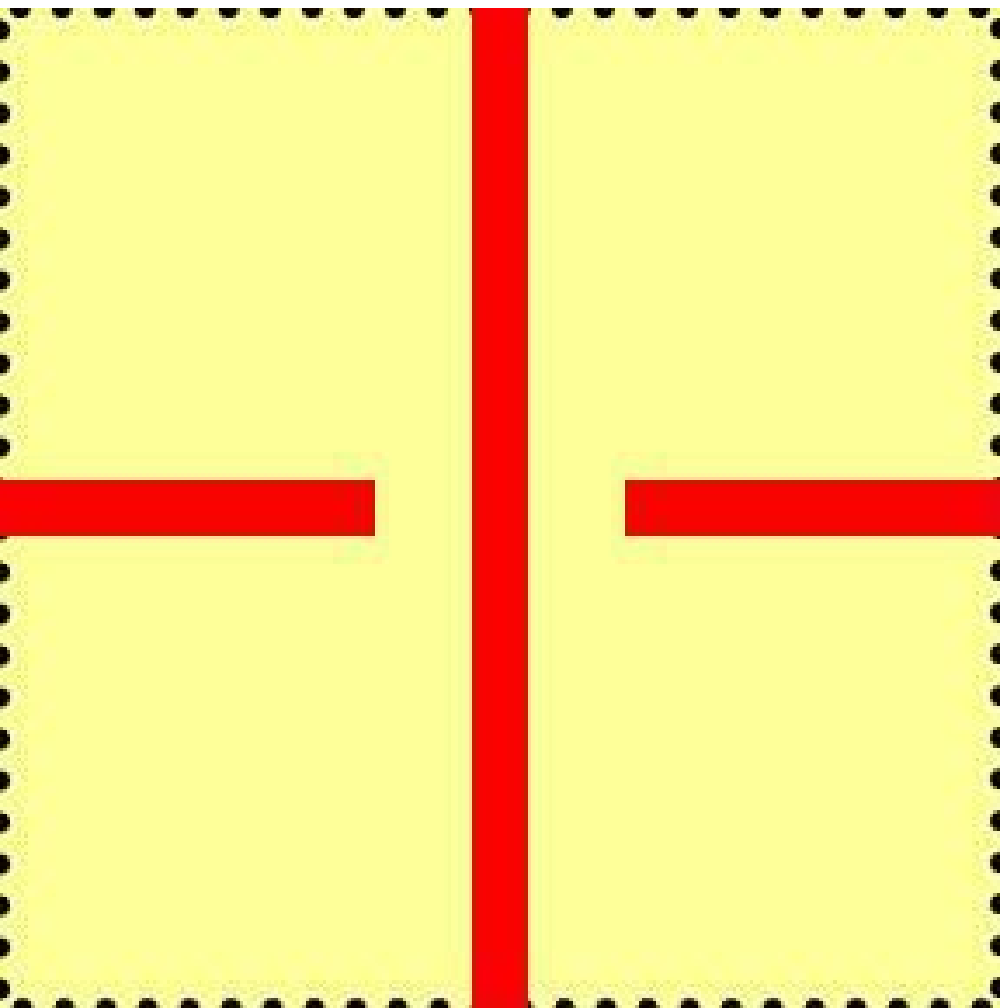} &\includegraphics[width=1cm]{ut01.EPS}  \\ 
\includegraphics[width=1cm]{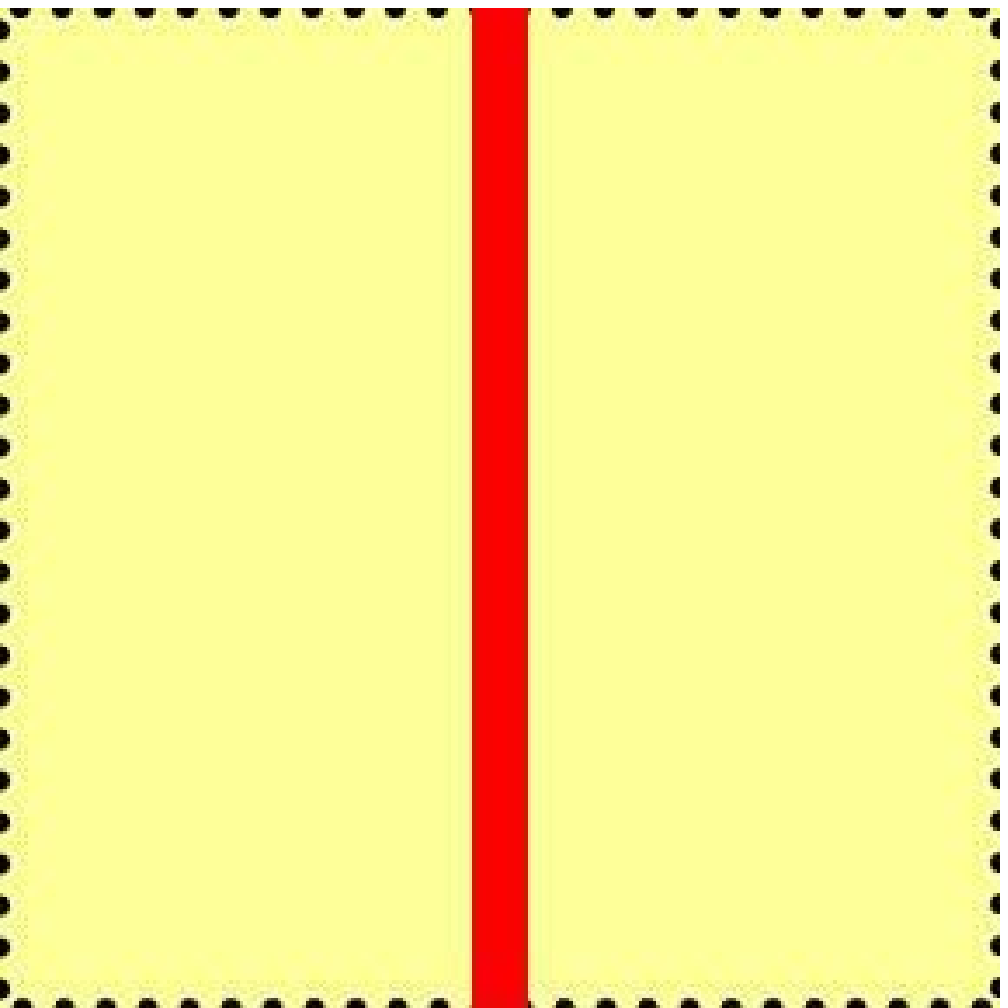} & \includegraphics[width=1cm]{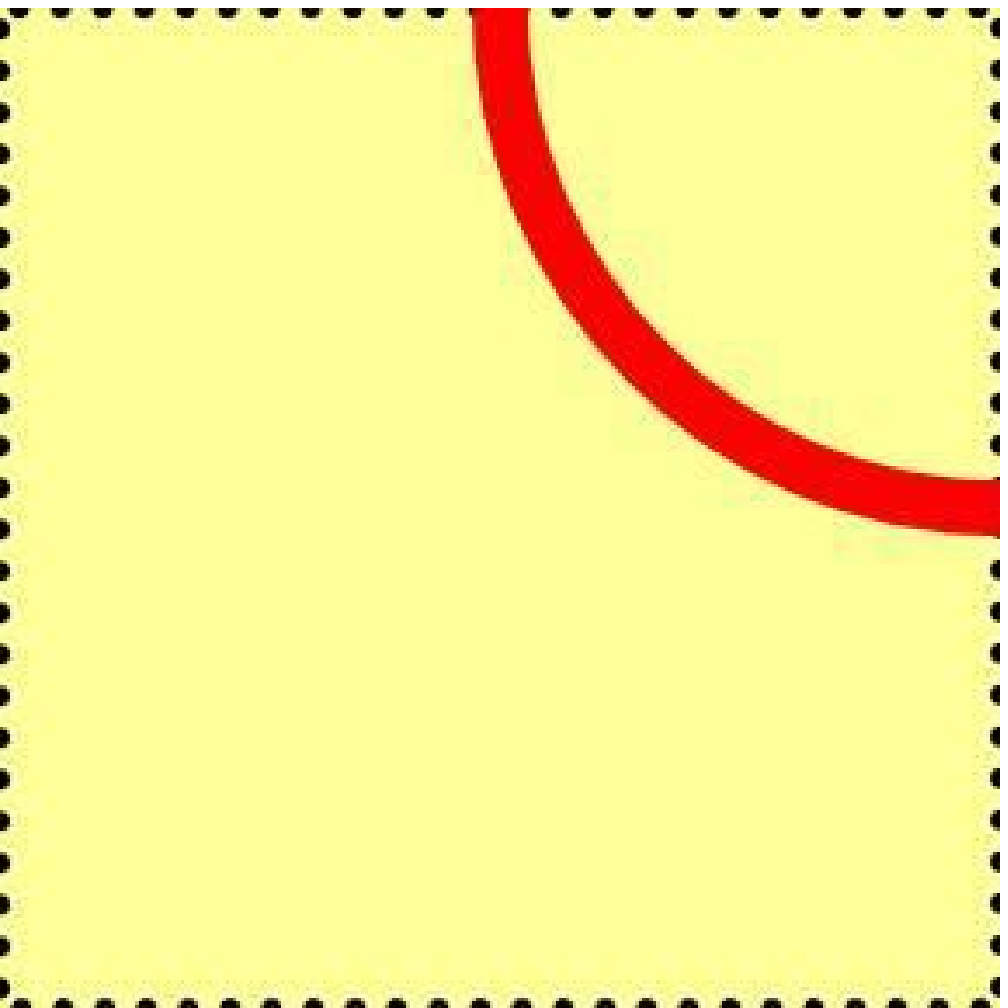} &
\includegraphics[width=1cm]{ut09.EPS} &\includegraphics[width=1cm]{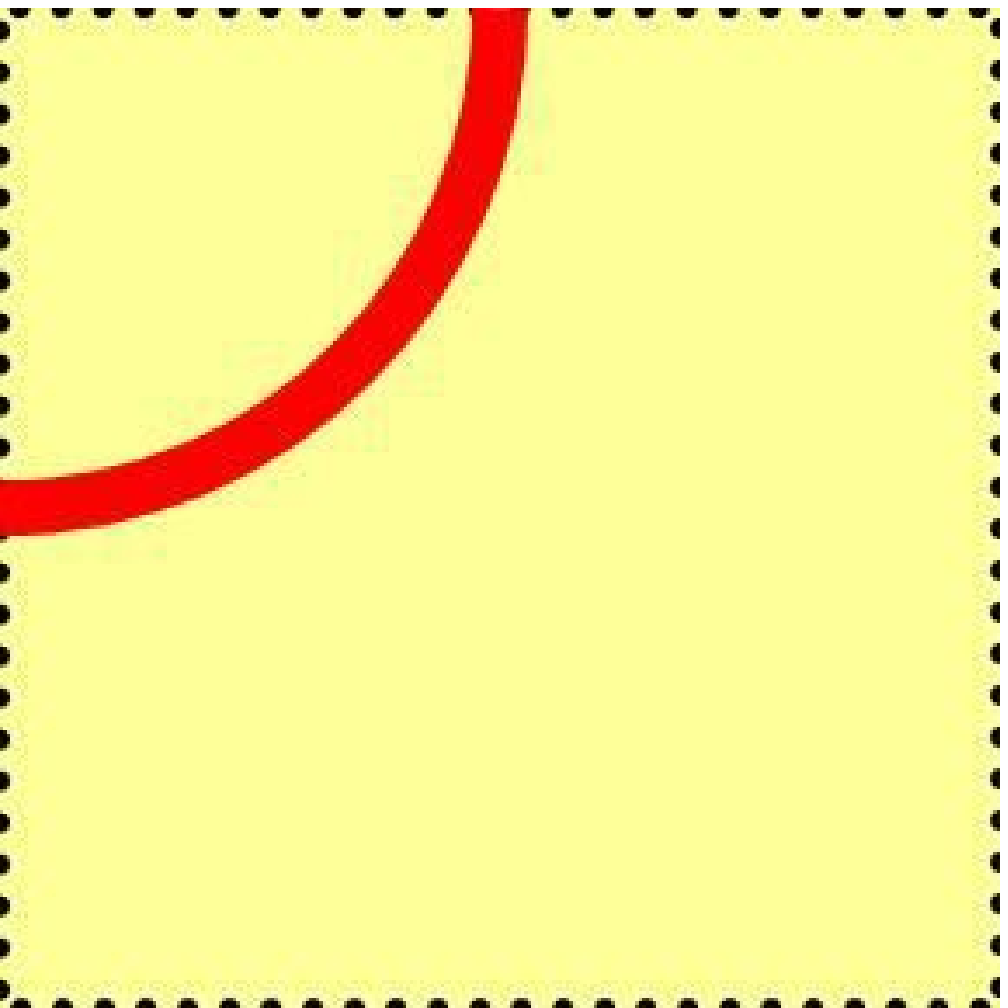}  \\ 
\includegraphics[width=1cm]{ut03.EPS} & \includegraphics[width=1cm]{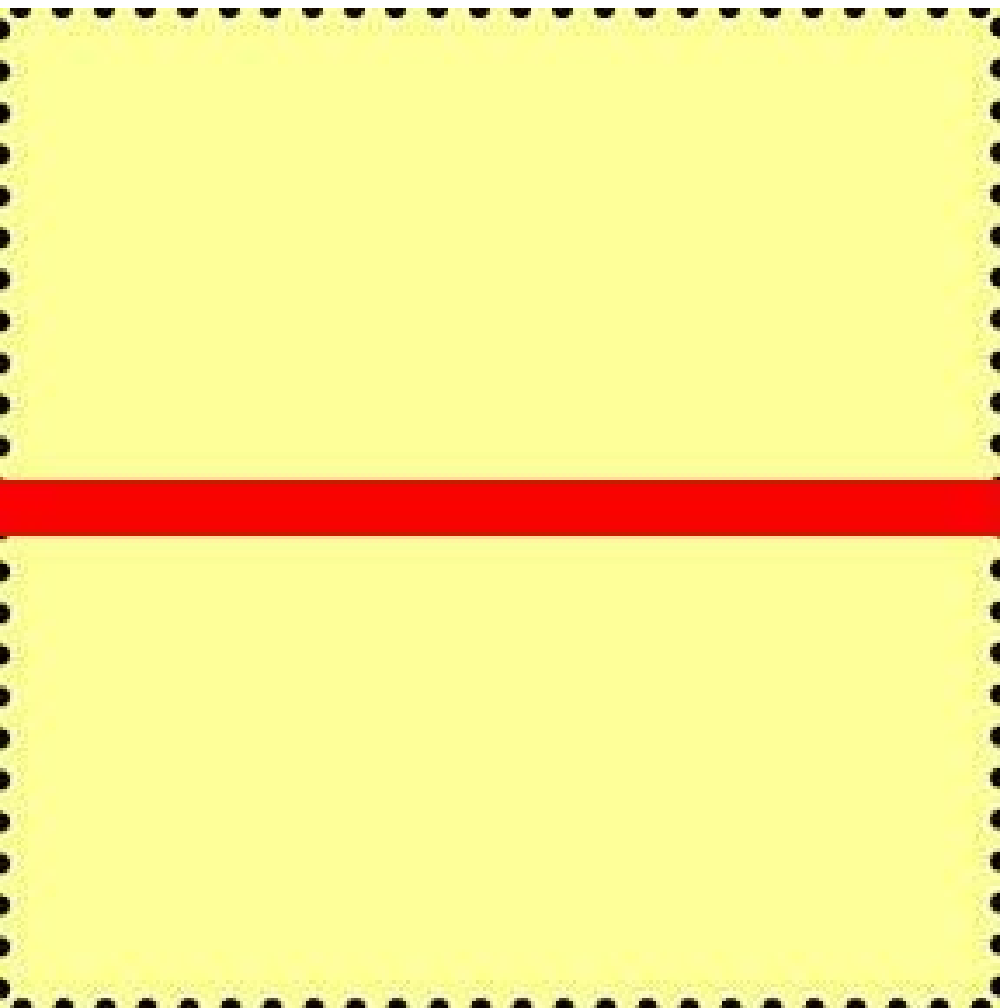} &
\includegraphics[width=1cm]{ut04.EPS} &\includegraphics[width=1cm]{ut00.EPS}   
\end{array}
$

$ 
\begin{array}{cccc}
\includegraphics[width=1cm]{ut00.EPS} & \includegraphics[width=1cm]{ut02.EPS} & 
\includegraphics[width=1cm]{ut01.EPS} &\includegraphics[width=1cm]{ut00.EPS}  \\ 
\includegraphics[width=1cm]{ut02.EPS} & \includegraphics[width=1cm]{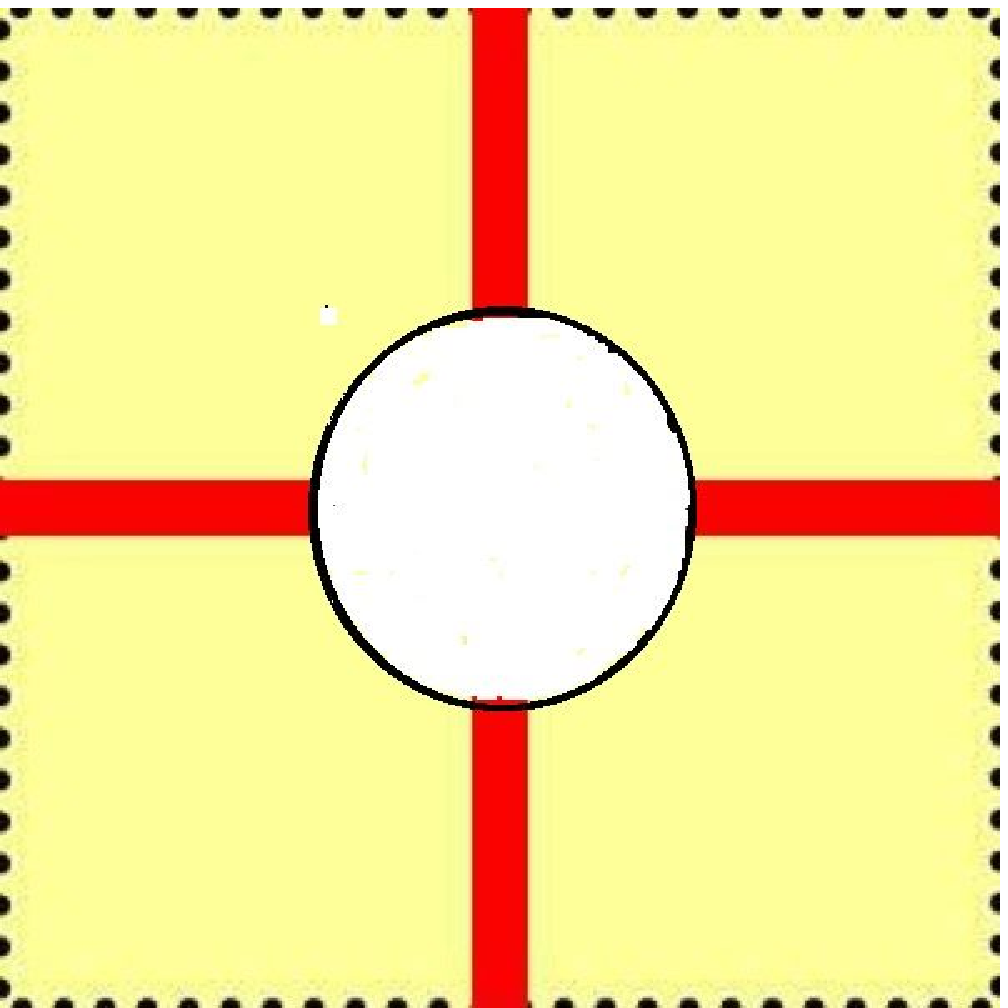} &
\includegraphics[width=1cm]{ut10.EPS} &\includegraphics[width=1cm]{ut01.EPS}  \\ 
\includegraphics[width=1cm]{ut06.EPS} & \includegraphics[width=1cm]{ut03.EPS} &
\includegraphics[width=1cm]{ut09.EPS} &\includegraphics[width=1cm]{ut04.EPS}  \\ 
\includegraphics[width=1cm]{ut03.EPS} & \includegraphics[width=1cm]{ut05.EPS} &
\includegraphics[width=1cm]{ut04.EPS} &\includegraphics[width=1cm]{ut00.EPS}   
\end{array}
$

$ 
\begin{array}{cccc}
\includegraphics[width=1cm]{ut00.EPS} & \includegraphics[width=1cm]{ut02.EPS} & 
\includegraphics[width=1cm]{ut01.EPS} &\includegraphics[width=1cm]{ut00.EPS}  \\ 
\includegraphics[width=1cm]{ut02.EPS} & \includegraphics[width=1cm]{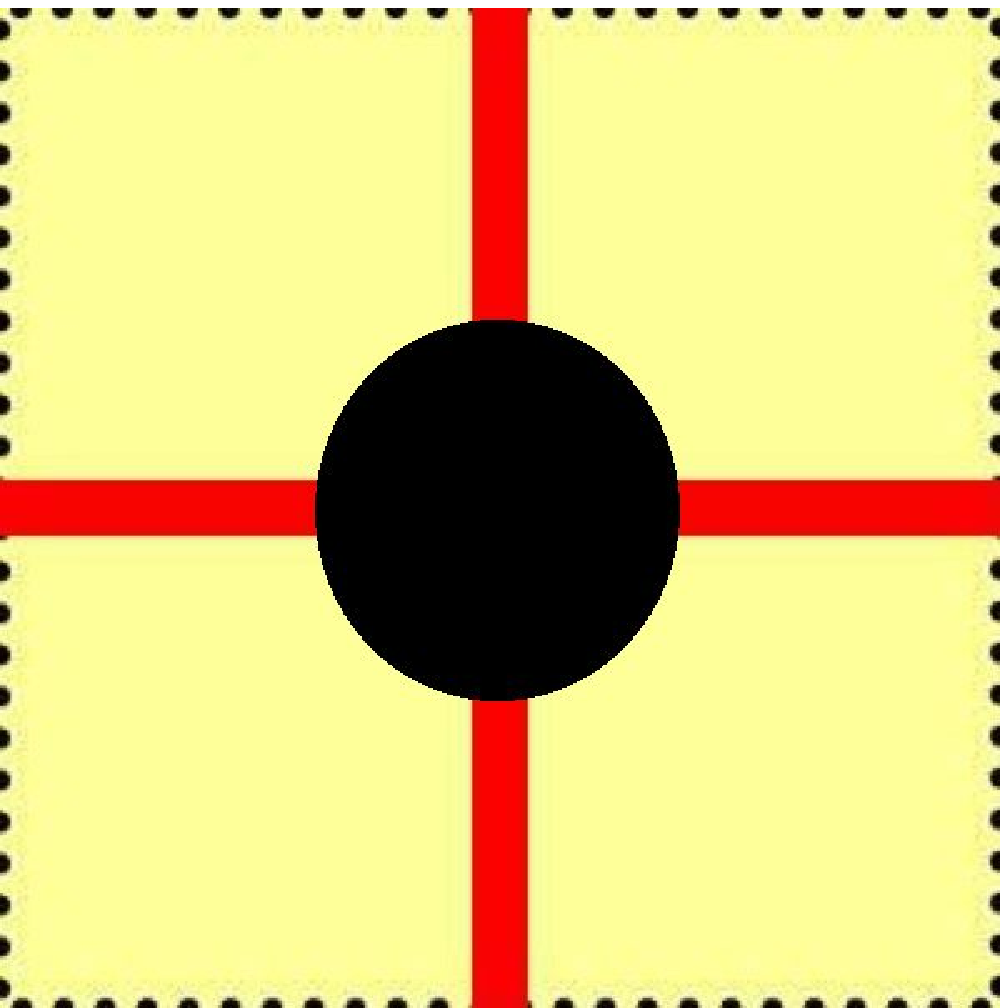} &
\includegraphics[width=1cm]{ut10.EPS} &\includegraphics[width=1cm]{ut01.EPS}  \\ 
\includegraphics[width=1cm]{ut06.EPS} & \includegraphics[width=1cm]{ut03.EPS} &
\includegraphics[width=1cm]{utvertex.EPS} & \includegraphics[width=1cm]{ut04.EPS}  \\ 
\includegraphics[width=1cm]{ut03.EPS} & \includegraphics[width=1cm]{ut05.EPS} &
\includegraphics[width=1cm]{ut04.EPS} &\includegraphics[width=1cm]{ut00.EPS}   
\end{array}
$

     \end{tabular}
     \caption{\bf Classical, Virtual and Graphical Mosaic Knots}
     \label{Figure 1}
\end{center}
\end{figure}

\begin{figure}
     \begin{center}
     \begin{tabular}{c}
     
$ 
\begin{array}{ccccccccccc}
\includegraphics[width=1cm]{ut01.EPS} & 
\includegraphics[width=1cm]{ut02.EPS} &\includegraphics[width=1cm]{ut03.EPS}  & 
\includegraphics[width=1cm]{ut04.EPS} &\includegraphics[width=1cm]{ut05.EPS} &
\includegraphics[width=1cm]{ut06.EPS}&\includegraphics[width=1cm]{ut00.EPS} &
\includegraphics[width=1cm]{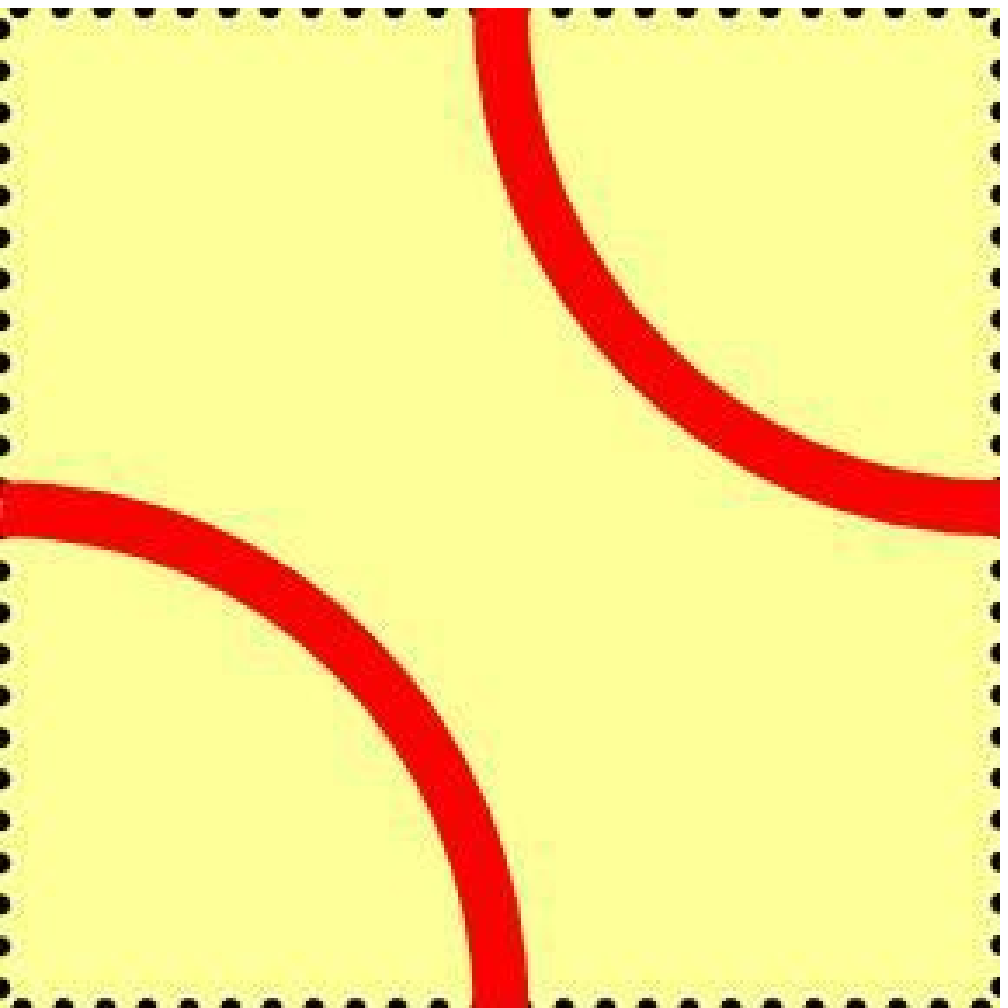}  & \includegraphics[width=1cm]{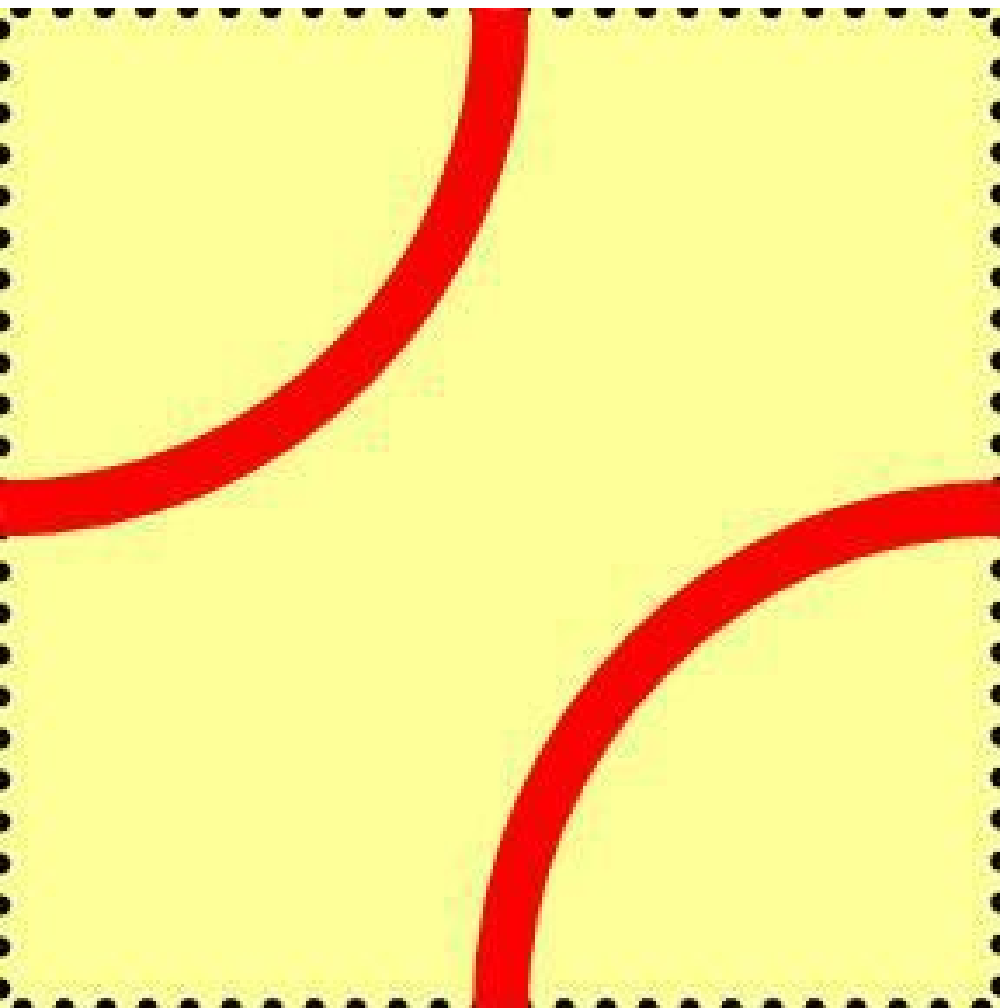} & 
\includegraphics[width=1cm]{ut09.EPS} &\includegraphics[width=1cm]{ut10.EPS}  
\end{array}
$

     \end{tabular}
     \caption{\bf Basic Tiles}
     \label{Figure}
\end{center}
\end{figure}


\section{General Quantization and Quantizing Classical Knots}
In this section we give a general definition of quantization, in analogy to that given in 
\cite{LomQknots2}. We then apply this definition to the quantization of classical knots that are represented by embeddings of the circle in Euclidean three-space.
\bigbreak

\noindent {\bf Definition.} Let $E$ be a collection of mathematical objects. We will call $E$ the set of 
{\em motifs} to be quantized (see \cite{LomQknots2}). Let $G$ be a group acting on the elements of 
$E$ so that each element of $G$ permutes $E.$ That is, we assume that for each $g \in G$ we have a
mapping taking $K \longrightarrow g(K)$ for each $K \in E$ such this is a 1-1 correspondence of 
$E$ with itself and so that $g(h(K)) = (gh)(K)$ where $g$ and $h$ are in $G$ and $gh$ denotes the product of these group elements in $G.$ We further assume that the identity element in $G$ acts as the identity mapping on $E.$ We then {\em quantize} the pair $(E,G)$ by forming a Hilbert space $H(E)$
with orthonormal basis consisting  in the set  $\{ |K\rangle : K \in E  \}.$  Here we take the elements of  
$H(E)$ to be finite sums of basis elements with complex coefficients and we use the usual Hermitian inner product on this space. Since the group $G$ acts on the basis by permuting it, we see that the action extends to an action of $G$ on $H(E)$ by unitary transformations. We call the new pair
$(H(E), G)$ (with this unitary action) the {\em quantization} of $(E,G).$
\bigbreak

Classical knot theory is formulated in terms of continuous embeddings of circles into the three dimensional space $R^{3}$ or the three dimensional sphere $S^{3}$ (which may be taken as the
set of vectors of unit length in Euclidean four dimensional space, or as the one-point compactification of 
$R^{3}.$ A knot is represented by an embedding $K: S^{1} \longrightarrow R^{3}.$ where $S^{1}$ denotes the circle (i.e. the set of points at unit distance from the origin in the Euclidean plane) with the 
topology inherited from the Euclidean plane. If $h:R^{3} \longrightarrow R^{3}$ is an orientation 
preserving homeomorphism of $R^{3}$, then by forming the composition $K' = h \circ K$ defined by
$h \circ K (x) = h(K(x))$ for $x \in S^{1}$, we obtain a new embedding $K'.$ We say that the two
embeddings $K$ and $K'$ are {\em equivalent}. We say that two embeddings $K$ and $K'$ represent the same knot type if there is an orientation preserving homeomorphism $h$ (as above) such that 
$K' = K\circ h.$ The set $G$ of orientation preserving homeomorphisms of $R^{3}$ forms a group under compositiion. The set of circle embeddings $$E(S^{1}) = \{K:S^{1} \longrightarrow R^{3}\}$$
is acted upon by $G$ via composition. In this way the  group $G$ acts as a group of permutations of the set $E(S^{1}).$ Note that we mean this action in the sense of group representations. We have that for
$g, h \in G,$ $$g\circ(h \circ K)) = (g\circ h)\circ K$$ for $K:S^{1} \longrightarrow R^{3},$ any embedding of a circle in $R^{3}.$ Note also that two elements $K$ and $K'$ of $E(S^{1})$ are equal if and only if they are point-wise equal as functions on the circle $S^{1}.$
\bigbreak

Let $H(E(S^{1})$ denote the Hilbert space for which the set of embeddings $E(S^{1})$ is an orthonormal basis. We take this space to be the set of finite linear combinations of its basis elements. We denote the basis elements of this Hilbert space by $|K \rangle$ where $K$ is an embedding of the circle in $R^{3}$. Using $G$ as defined above, we have $G$ applied to the basis elements of $H(E(S^{1})$ acting as a group of permutations of the basis. These permuations extend to unitary transformations on the entire Hilbert space, giving a quantization of $(E(S^{1}, G)$ in accord with the definition given in this section.
\bigbreak

\noindent {\bf Remark.} See Figure 3 for an illustration of a classical knot equivalence. Note that the
quantization of the set embeddings that represent classical knots gives a Hilbert space of uncountable 
dimension, just as there are an uncountable number of embeddings that can represent knots in three dimensional space. Thus this quantization must be contrasted with the mosiac knots where we have created a hierarchy of finite dimensional spaces and finite groups to handle quantum information for combinatorial knot theory. The quantization of classical knots that is given in this section is 
intellectually satisfying since it quantizes the full geometrical context for knot theory. This same 
context of embeddings of objects or placements of structures in three dimensional space is the place where most ideas in geometry, topology and physics are carried out. Thus we expect that this very large quantization of knots will be useful in studying knots in physical situations such as vortices in 
super-cooled helium \cite{Rasetti} or the possibility of knotted structures in gluon fields \cite{BK}.
\bigbreak 

\begin{figure}[htb]
     \begin{center}
     \begin{tabular}{c}
     \includegraphics[width=7cm]{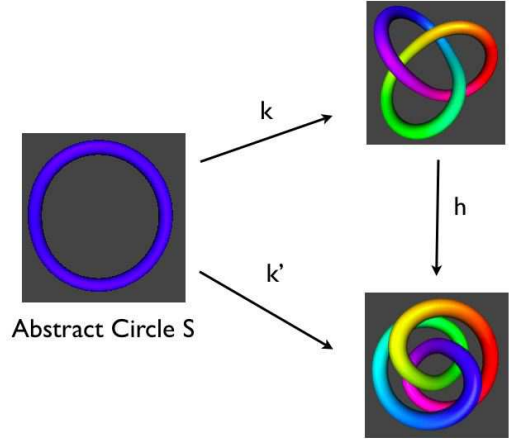}
      \end{tabular}
     \caption{\bf Classical knot equivalence via ambient homeomorphism}
     \label{Figure}
\end{center}
\end{figure}


\section{Classifying Maps for Knots in the Three-Sphere}
In this section we give a short proof of the well-known theorem described below. The theorem tells us that any knot in the three-sphere (or equivalently in Euclidean three-space)
can be represented as the inverse image of $0$ for a differentiable mapping from $S^3$ to $D^2.$ We can include $D^2$ in the complex plane $C$ and view the map as a time-independent  wave function. Thus this general theorem tells us that that any knot can be regarded as the set of zeroes of a quantum wave function. Recent results since Berry
\cite{} have given many realizations of such knotted zeroes as algebraic varieties and with regard to specific physical systems such as the hydrogen atom and the harmonic oscillator.
Here we point out the generality of this phenomenon, and we shall ask many questions in the discussion below. For now we point out that this Theorem, allows us to shift the definition of 
quantum knot to this category of maps from the three-sphere to the complex plane, and to view a knot as the zeroes of a wavefunction. Quantum knots are knotted zeroes of time-independent wave functions.\\

\begin{theorem}  Given a smooth knot $K$ in $S^3$ (the three-dimensional sphere),  there exists a differentiable map $f:S^3 \longrightarrow D^2$ so that $f$ is transverse to $0 \in D^2$ (maximal rank) and $f^{-1}(0) = K.$
\end{theorem}\\

\begin{proof} Let $K$ be described by a knot diagram with arcs $\{ a_1,a_2, \cdots , a_{n} \}.$ Each crossing $i$ in the diagram has an overcrossing arc $a_{over(i)}$ and two incident
undercrossing arcs $a_{input(i)}$ and $a_{output(i)}.$ Input and output arcs are chosen so that  the input arc sees the overcrossing line going from left to right as one approaches the crossing along the input arc. This direction of approach may or may not agree with the orientations on the undercrossing arcs.
In the Wirtinger presention of the fundamental group of the knot complement $S^{3} - K$ there is an associated relation of the form  
$$a_{output(i)} = (a_{over(i)})^{-1}  a_{input(i)} a_{over(i)}.$$
It is known (See \cite{Stillwell}.) that the homotopy type of the knot complement $S^3 - K$ is the same as that of the $CW$ complex $X(K)$ formed by taking a wedge of circles, one circle for each
of the generators $a_{i}$ and  and one two-cell for each relation as described above, and a three-cell that is attached to these 2-cells as we shall describe below. The boundary of the two-cell is
$$(a_{output(i)})^{-1}(a_{over(i)})^{-1}  a_{input(i)} a_{over(i)}.$$
Letting $c= a_{output(i)},$  $b = a_{over(i)}$ and $a = a_{input(i)},$ we have  the relation
$$c = b^{-1} a b$$ and cell boundary $$c^{-1} b^{-1} a b.$$\\

With this description, form the map $F: \{ a_1,a_2, \cdots , a_{n} \} \longrightarrow S^{1}$ by taking each $a_{i}$ in the wedge of circles to $S^{1}$ diffeomorphically according to its orientation. Letting $t$ denote the generator of the fundamental group of $S^{1},$ this means that at the level of the fundamental group each $a_{i}$ is sent by $F$ to $t$, and hence the boundary of the two-cell (in every case) is sent via  $c^{-1} b^{-1} a b \longrightarrow t^{-1} t^{-1} t t  = 1.$ Hence the mapping $F$ extends over each two-cell in the complex $X(K),$ and since there is no obstruction to extending over the three-cell,  the map $F:X(K) \longrightarrow S^{1}$ is now defined.\\

The complex $X(K),$ described in the last paragraph, is obtained as follows. Take a diagram for the knot $K$ that is almost planar, with arcs above the plane for each overcrossing in the diagram, each such arc meeting the plane in two points and then continuing below the plane a short distance to form corresponding horizontal arcs for the undercrossings.
Take a tubular neighborhood of the knot so that the intersection of the tubular neighborhood with the plane is a disjoint collection of discs that result from the thickenings of the undercrossing arcs. Removing the tubular neighborhood of the knot, one has that the space $A$ above the plane has free fundamental group generated by the arcs 
$\{ a_1,a_2, \cdots , a_{n} \}$ and has the homotopy type of a wedge of circles, one for each $a_{i}.$ The space $B$ below the plane is a three-ball, and the intersection of $B$ with
the plane is a punctured plane with one hole for each crossing and such that the element of fundamental group carried by the boundary of the hole is $(a_{output(i)})^{-1}(a_{over(i)})^{-1}  a_{input(i)} a_{over(i)}$ at that crossing. Thus the boundary of this hole bounds a disc in the lower space $B,$ and the van Kampen Theorem gives us the above result about the fundamental group. As we see from this description, the complex $X(K)$ has the homotopy type of the complement of the tubular neighborhood of the knot with the two-cells partially embedded in the lower part $B$ and extending into $A$ to bound the appropriate loops. The three-cell is attached via the lower part $B.$\\

Choose a tubular neighborhood of the knot $K$ so that we have an embedding of $S^1 \times D^2$ into $S^3$ with image $N(K) = K \times D^2.$ One can choose a system of circles on the boundary of this tubular neighborhood that correspond to the Wirtinger generators so that the two-cells of $X(K)$ are embedded in $S^3 - Interior(N(K))$ and the rest of $S^3 - Interior(N(K))$ retracts to $X(K).$ In this way we obtain a mapping $f: S^{3} - Interior(N(K)) \longrightarrow S^{1}$ extending our mapping $F:X(K) \longrightarrow S^1$ and such that 
the restriction of $f$ to the boundary of $N(K)$ is the projection $K \times S^1 \longrightarrow S^1.$ This map extends to $K \times D^2 \longrightarrow D^2$ by taking projection to the second factor. With this extension we have constructed a mapping $$f: S^3 \longrightarrow D^2$$ such that $f^{-1}(0) = K.$ This completes the proof of the Theorem. 
\end{proof}

\noindent {\bf Remarks} Note that the Theorem works for links as well as for knots, with the caveat that one must consider tubular neighborhoods of each link component.
This Theorem is sometimes proved by noting that the first cohomology group of a space $X$ is given by the homotopy classes of mappings of that space to the circle: $$H^{1}(X) = [X,S^1]$$ where $[X,Y]$ denotes the set of homotopy classes of maps from $X$ to $Y.$ Thus  $H^{1}(S^3 - K) = [S^3 - K, S^1]$ and the proof follows by using the fact that $H^{1}(S^3 - K) = Z,$ the integers. Here we have given an explicit construction of the mapping up to homotopy. We can think of the mapping $f:S^3 \longrightarrow D^2$ as a map into the complex numbers $C,$ and hence as a time independent wave function whose zeroes are the knot. This leads to a number of good questions and further
remarks:
\begin{enumerate}

\item  Given $f: S^3 \longrightarrow C$ representing a knot $K \subset S^3$ so that $f^{-1}(0) = K,$ and a homeomorphism $h: S^3 \longrightarrow S^3$ with $K = h(K'),$ we have
$g = f \circ h:  S^3 \longrightarrow C$ with $g^{-1}(0) = K'.$ Thus the ambient group of diffeomorphisms of the three-sphere acts on the knots represented as wave-functions. We can regard these as unitary actions of this group with an appropriate quantum space for the ``quantum knots" $f: S^3 \longrightarrow C.$ This point of view needs investigation.

\item Given a grid diagram or a mosaic quantum knot, how concretely can we construct the mapping f? We have described in the proof of the theorem how to 
make the mapping from a diagram, and hence from a mosaic diagram. This construction involves choices that pinpoint it only up to homotopy type. It may be possible to choose a more canonical method to produce the classifying mapping.

\item Given a knot $K$ in $S^3$ one can construct a Seifert spanning surface $F$ for $K,$  a surface embedded in the three-sphere. By splitting $S^3 - N(K)$ along $F$ one obtains a three-manifold $M$ with boundary $F_{-} \cup F_{+}$ a union of two copies of $F$ whose intersection is exactly $K.$ An argument similar to the one given in the Theorem above, shows that there is a mapping $F: M - (K \times I) \longrightarrow I$ where $I = [0,1]$ with $F^{-1}(0) = F_{-}$ and $F^{-1}(1) = F_{+}.$ This closes to a map from $S^{3} - N(K)$ to the circle, and again yields a classifying map for the knot $K.$ This construction via Seifert surface has particularly good properties in the case of a fibered knot, where the the mapping $F$ is a fiber bundle over the interval, and the corresponding mapping $f:S^3 - N(K) \longrightarrow S^1$ is a fiber bundle over the circle with each fiber a copy of the spanning surface for the knot. 

\item In \cite{Milnor} Milnor proves a fibration theorem for knots that are links of complex hypersurface singulaties. There one starts with a complex polynomial mapping $f:C^{n} \longrightarrow C$ with an isolated singularity at the origin. The { \it link} of the singularity $L(f)$ is given by the formula $L(f) = Var(f)\cap S^{2n-1}_{\epsilon}$ where $Var(f)$ denotes the collection of solutions to the equation $f(z_1,z_2,\cdots, z_n) = 0$, the variety of $f$, and $S^{2n-1}_{\epsilon}$ is a sufficiently small sphere about the origin in $C^{n}.$ The resulting $L(f)$ is called the link of the singularity and it is a manifold of dimension $2n-3,$ hence of codimension two in the sphere. When $n=2$ we have that $L(f)$ is a link or knot in the three-sphere. Milnor shows that the mapping $\phi = f/|f| : S^{2n-1}_{\epsilon} - L(f)  \longrightarrow S^{1}$ is a smooth fibration of the complement of $L(f).$ As we have described above, this fibration allows us to construct a classifying map $F: S^{2n-1} \longrightarrow C$ such that $F^{-1}(0) = L(f).$ Thus we obtain specific classifying mappings for knots and links that are associated with isolated complex polynomial singularities. Similar results can be obtained in special cases by using only partially polynomial maps. For example in 
\cite{Rudolph} Lee Rudloph shows that the if $F(z,w) = w^3 - 3z \bar{z}(1 + z + \bar{z})w - 2(z + \bar{z})$ and $G(z,w) = F(z^2,w)$ then the figure eight knot is the link of $G.$ One can see directly that the figure eight knot is a fibered knot by using the branched covering methods of Goldsmith \cite{Goldsmith} and generalized by  Harer \cite{Harer}, but the Rudolph construction gives explicit semi-analytic polynomial mappings that classify the knot. More along these lines in relation to the physics can be found in the work of Berry \cite{Berry} and Dennis \cite{BodeDennis} and their collaborators. Furthermore the papers by Peralta-Salas and collaborators \cite{Peralta1,Peralta2} establish that for any finite link there exists (complex-valued) eigenfunctions of the harmonic oscillator in $R^3$ (and also for the Coulomb potential) such that the link is a union of connected components of the nodal set of the 
eigenfunction.

 \item For fibered knots, we have that the classifying map is a fibration whose fibers are all diffeomorphic copies of a spanning surface for the knot. In some cases we can construct
the fibration quite explicity (e.g for the figure eight knot) via lifting  branched covering constructions as pioneered by Goldsmith \cite{Goldsmith}. This suggests starting with these as test cases for Schrodinger evolution. We would like the classifying mapping $f$ to be regarded as an initial condition for a Schrodinger evolution of $f$ as a wave function. 
Then we can ask: How does the zero set behave under the Schrodinger evolution? How is this related to the work of Berry \cite{Berry} and other physicists and mathematicians working on this problem?

\item In thinking about the evolution of the zeroes of the wave function, recall that the 
Schrodinger equation has the form  $$i \hbar \partial \psi / \partial t  = H \psi$$ where $H$ is the Hamiltonian operator.
Suppose that $$\psi = \psi(x,0) : R^3 \longrightarrow C$$ has $\psi^{-1}(0) = K,$ a knot or link in $R^3.$
What can we say about the temporal evolution of the inverse image of $0?$ What Hamiltonians should we examine?\\

We could take $$\psi = e^{-i/\hbar H t} \psi_{0}$$
if $H$ is time independent. 
Lets remember that in Mosaic Quantum Knots we have the ambient group acting on them unitarily. And similarly we have an ambient group of diffeomorphisms acting on maps
$S^3 \longrightarrow C$ by composing with diffeos of $S^3$ that carry $K'$ to $K.$ We can make a Hamiltonian that corresponds to an isotopy of one knot to another and place it in a Schrodinger evolution. This at least can be said. 

\item This paper has been designed to set the stage for an exploration of these issues.

\end{enumerate}

\end{document}